\documentclass[12pt]{article}

\usepackage{a4wide}
\usepackage{amsmath}
\usepackage{amsthm}
\usepackage{amssymb}
\usepackage[english]{babel}
\usepackage{graphicx}

\newcommand{\R}{\mathbb{R}}

\newcommand{\Z}{\mathbb{Z}}
\newcommand{\Sd}{\textup{sd}} 
\newcommand{\st}{\textup{st}} 
\newcommand{\ord}{\textup{Ord}}

\DeclareMathAlphabet{\mathcalligra}{T1}{calligra}{m}{n}

\DeclareFontFamily{U}{mathx}{\hyphenchar\font45}
\DeclareFontShape{U}{mathx}{m}{n}{
      <5> <6> <7> <8> <9> <10>
      <10.95> <12> <14.4> <17.28> <20.74> <24.88>
      mathx10
      }{}
\DeclareSymbolFont{mathx}{U}{mathx}{m}{n}
\DeclareFontSubstitution{U}{mathx}{m}{n}
\DeclareMathAccent{\widecheck}{0}{mathx}{"71}

\newtheorem{theo}{Theorem}[section]
\newtheorem{lemma}[theo]{Lemma}
\newtheorem{prop}[theo]{Proposition}
\newtheorem{cor}[theo]{Corollary}
\newtheorem{rem}[theo]{Remark}
\newtheorem{ex}[theo]{Example}

\newtheorem{defi}[theo]{Definition}

\begin{document}
\title{Shellable tilings on relative simplicial complexes and their $h$-vectors\thanks{This work was relatively supported by the ANR project MICROLOCAL (ANR-15CE40-0007-01).}}
\author{Jean-Yves Welschinger}
\maketitle

\begin{abstract} 
\vspace{0.5cm}

An $h$-tiling on a finite simplicial complex is a partition of its geometric realization by maximal simplices deprived of several codimension one faces together with possibly their remaining face of highest codimension. In this last case, the tiles are said to be critical. An $h$-tiling thus induces a partitioning of its face poset by closed or semi-open intervals. We prove the existence of $h$-tilings on every finite simplicial complex after finitely many stellar subdivisions at maximal simplices. These tilings are moreover shellable. We also prove that the number of tiles of each type used by a tiling, encoded by its $h$-vector, is determined by the number of critical tiles of each index it uses, encoded by its critical vector. In the case of closed triangulated manifolds, these vectors satisfy some palindromic property. We finally study the behavior of tilings under any stellar subdivision. \\


\end{abstract}

\section{Introduction}

A finite simplicial complex $K$ is classically said to be {\it shellable} when its maximal simplices $\sigma_1 , \dots , \sigma_N$ can be totally ordered in a such way that for every $p \in \{1 , \dots , N-1 \}$,
$\sigma_{p+1} \setminus (\sigma_1 \cup \dots \cup \sigma_p)$ is a non-empty union of codimension one faces of $\sigma_{p+1}$, see \cite{Bing,Bjo,Koz,Z}. It has been proved by H. Bruggesser and  P. Mani \cite{BruMan} that the boundary of any convex polytope is shellable. This led to the upper bound theorem by P. McMullen \cite{McMu}, up to which cyclic polytopes have maximal number of faces in each dimension, and to the complete characterization of face vectors of convex polytopes by L. J. Billera-C. W. Lee \cite{bilLee} and R. Stanley \cite{Stan}, the g-theorem, see also \cite{Z,Fult}. The property is strong though, a shellable closed triangulated manifold has to be piecewise-linear and homeomorphic to a sphere \cite{Koz}, and many triangulated spheres are not shellable even in dimension three \cite{Lick,HachZ}, even though, when piecewise-linear, they all become after finitely many barycentric subdivisions \cite{AdipBen2}. It has been relaxed in several ways, including collapsibility \cite{Whi,Lick,AdipBen,Tan,RS}, a combinatorial counterpart to contractibility, constructibility \cite{Stan75,Hach}, local constructibility \cite{DurJ,BenZ} or partitionability \cite{KleinSmi,Stanbook,Z}, but all these properties remain restrictive with respect to the underlying topology of the complex. We recently introduced a geometric counterpart to the latter, observing likewise \cite{SaWel1} that any shelling on $K$ provides in particular a tiling of its geometric realization by basic tiles which are maximal simplices, or {\it facets}, deprived of several codimension one faces, or {\it ridges}. Every product of a sphere with a torus of positive dimension carries such a tileable triangulation which cannot be shelled by Theorem $1.1$ of \cite{Wel1}.  A basic tile of {\it order} $k$, that is having been deprived of $k$ ridges, contains a unique open face of dimension $k-1$, its {\it restriction set} \cite{Stanbook,Z,BjoW}, and we define \cite{SaWel2} a {\it critical} tile of {\it index} $k$ to be the one obtained after removing also this peculiar face. The face structure of such a critical tile $C$ is thus a semi-open interval $]r(C),C]$, where $r(C)$ denotes the removed restriction set, see \S \ref{subsecgeompartition}. This terminology originates from their relation with the discrete Morse theory of R. Forman \cite{For1,For2} established in \cite{SaWel2} for more general {\it Morse tiles} which are closed simplices deprived of several ridges together with possibly a unique face of higher codimension, its {\it Morse face}. We now define an {\it $h$-tiling} of a finite simplicial complex $K$, or more generally of a relative simplicial complex, see \S \ref{subsectiling}, to be a partition of its geometric realization $\vert K \vert$ by either basic or critical tiles such that for every $d \geq 0$, the union of tiles of dimension $\geq d$ is closed in $\vert K \vert$. It thus induces a partition of the relative complex $K \setminus \{ \emptyset \}$ by closed or semi-open intervals, one for each facet of $K$, see \S \ref{subsecgeompartition}. It is called a {\it Morse tiling} when it uses Morse tiles instead, and these tilings are said to be {\it shellable} when the tiles $T_1 , \dots , T_N$ can be totally ordered in such a way that for every $p \in \{1 , \dots , N\}$,
$T_1 \cup \dots \cup T_p$ is closed in $\vert K \vert$. 

Our main results are the following tiling and existence theorems, Theorems \ref{theotiling} and \ref{theostellar}.
\begin{theo}
\label{theotiling}
Let $T$ be a Morse tile, $\mu \subset \overline{T}$ its Morse face and $\tau \subset \overline{T}$ any face of positive dimension. Then, if $\tau \not\subset \mu$ (resp. $\tau \subset \mu$), the stellar subdivision of $T$ at $\tau$ carries a Morse shelling using only regular basic tiles (resp. regular Morse tiles) together with a unique tile isomorphic to $T$. If $\emptyset \neq \mu \subsetneq \tau$, it also carries a Morse shelling using two critical tiles of consecutive indices $\dim (\mu)$ and $\dim (\mu) + 1$ together with regular basic tiles and $( \dim (\mu) - \ord (T))$ regular Morse tiles having $( \dim (\mu) - 1)$-dimensional Morse faces and order ranging from $\ord (T)$ to $\dim (\mu) - 1$.
\end{theo}

In Theorem \ref{theotiling}, $\overline{T}$ denotes the underlying simplex of $T$, see \S \ref{subsectiling}. We deduce that the $h$-tileability or shellability property gets preserved under stellar subdivision at any facet or ridge, see also Corollary \ref{cortiling} and Proposition \ref{propMorse}.
\begin{cor}
\label{cortilingintro}
The stellar subdivision at any facet and any ridge of an $h$-tiled or $h$-shelled (resp. Morse tiled or Morse shelled) finite relative simplicial complex inherits an $h$-tiling or $h$-shelling (resp. Morse tiling or Morse shelling) using the same number of critical tiles with same indices, while any Morse tiling (resp. Morse shelling) inherits an $h$-tiling (resp. $h$-shelling) after finitely many stellar subdivisions at facets. Moreover, the latter holds true using stellar subdivisions at ridges instead, or also using mixed ones.
\end{cor}
In fact, every finite simplicial complex, in particular every closed manifold or pseudo-manifold being piecewise linear or not, carries some shellable $h$-tiling after finitely many stellar subdivisions at facets or ridges.  

\begin{theo}
\label{theostellar}
Every finite relative simplicial complex becomes Morse shellable after a single barycentric subdivision and carries a shellable $h$-tiling after finitely many stellar subdivisions at facets.
Moreover, the same holds true using stellar subdivisions at ridges instead, or also using mixed ones.
Finally, in bounded dimension, both the sequence of subdivisions and the shelling are given by some quadratic time  algorithm.
\end{theo}
The complexity of the algorithm used to prove Theorem \ref{theostellar} is measured in terms of the size if the complex, given by the number of its facets.
In the case of a single non-basic regular Morse tile of index $k$ with $l$-dimensional Morse face for instance, the shellable $h$-tiling of Theorem \ref{theostellar} or Corollary \ref{cortilingintro} is obtained after $2^{l-k}$ stellar subdivisions and produces an $h$-shelling using $2^{l+1-k}$ critical tiles, see Proposition \ref{propMorse}.
This result is slightly refined in the case of closed triangulated pseudo-manifolds, see Theorem \ref{theostellarman} and \S \ref{subsecprelim} for a definition. Recall that D. E. Sanderson  \cite{Sand} likewise proved that any triangulated three-cell has a shellable subdivision and K. A. Adiprasito and B. Benedetti  \cite{AdipBen2} that every piecewise linear sphere becomes shellable after finitely many barycentric subdivisions. 

Every shellable Morse tiling on a finite simplicial complex encodes a class of compatible discrete Morse functions whose critical points are in one-to-one correspondence with the critical tiles of the tiling, preserving the index, see \cite{SaWel2} and Remark \ref{remspectral}, so that Theorem \ref{theostellar} provides a class of non-trivial discrete Morse functions on all finite simplicial complexes after finitely many stellar subdivisions at facets or ridges. 
Also, every Morse shelling of a finite relative simplicial complex given by Theorem \ref{theostellar} provides two spectral sequences which converge to its relative (co)homology and whose first pages are spanned by the critical tiles of the tiling, see \cite{Wel2}. Combined with Theorem \ref{theostellar}, they provide a way to compute the relative (co)homology of finite relative simplicial complexes using (co)chains complexes of much lower dimensions than the simplicial ones, as the discrete Morse complexes would do.

For every finite relative simplicial complex $K$, we denote by $s(K)$ the minimal number of stellar subdivisions at facets required for it to carry a shellable $h$-tiling, see Definition \ref{defsK}, so that $0 \leq s(K) < +\infty$ by Theorem \ref{theostellar}. What is the algorithmic complexity of deciding whether this stellar complexity vanishes? Recall that deciding collapsibility \cite{Tan} or shellability \cite{Goaoc} is NP-complete, while contractibility is undecidable \cite{VKF,Tan}.
When $s(K) = 0$, we likewise define $\mu (K)$ to be the minimal number of critical tiles of the shellable $h$-tilings on $K$, see Definition \ref{defmuK} and $\mu_\infty (K) \leq \mu (K)$ to be the infimum of these Morse numbers over all complexes obtained from $K$ after finitely many stellar subdivisions at facets, see Definition \ref{defmuinftyK}. They measure the lack of shellability in the classical sense.
\begin{prop}
\label{propmuK}
Let $K$ be a closed $n$-dimensional pseudo-manifold such that $s(K)=0$ and $n \geq 1$. Then, $\mu(K) \geq 2$ with equality if and only if $K$ is shellable. 
\end{prop}
By the discrete Morse theory of R. Forman \cite{For1}, any discrete Morse function on such a closed pseudo-manifold has likewise two critical faces or more and exactly two only if the manifold is collapsible once deprived of a facet.
\begin{prop}
\label{propknotintro}
Let $K$ be a locally constructible triangulated three-sphere which contains a knotted triangle in its one-skeleton. Then, 
there exists a discrete Morse function on $K$ having only two critical points, while $\mu_\infty (K) \geq 4$.
\end{prop}
We refer to \cite{BenZ} for a definition of local constructibility and to Proposition \ref{propknot} for a slightly more precise result.
By Theorem $4.4$ of \cite{JosPfe}, it is strongly NP-complete to decide whether, given $c>0$, a simplicial complex carries a discrete Morse function with less than $c$ critical points. What about the complexity of computing $\mu(K)$ or $\mu_\infty (K)$, or of deciding, given $\mu >0$, whether $\mu(K)$ or $\mu_\infty (K)$ are less than $\mu$? We likewise define $m_\infty (K)$ to be the infimum over all $d>0$ and all Morse shellings on the $d$-th barycentric subdivision $\Sd^d (K)$ of $K$ of the number of critical tiles it uses, see Definition \ref{defm}. It is positive by Theorem \ref{theotiling} and equals two for all piecewise linear triangulated spheres by Theorem B of \cite{AdipBen2}, see also Proposition \ref{propm}. How does this number compare with the minimal number of critical points of a discrete Morse function on $\Sd^d (K)$, $d \gg 0$?

These tilings are also related to the classical theory of $h$-vectors, see \cite{Stan1,Stan2,Z}. Following \cite{SaWel1,Wel1}, we define the {\it $h$-vector} $h(\tau)=(h_0(\tau), \dots , h_{n+1}(\tau))$ of a Morse tiling $\tau$ on an $n$-dimensional finite relative simplicial complex to be the vector whose $j$-th entry is the number of tiles of order $j$ used by $\tau$ and its {\it critical} or {\it $c$-vector} 
$c(\tau)=(c_0(\tau), \dots , c_{n}(\tau))$ to be the vector whose $j$-th entry is the number of critical tiles of index $j$ used by $\tau$, see \S \ref{sechvectors}.
It turns out that, in the case of $h$-tilings, one vector determines the other by the following uniqueness result which provides a second advantage of $h$-tilings with respect to Morse ones, the first one being that much less isomorphism types of tiles get involved in the tilings.
\begin{theo}
\label{theohvector}
Let $\tau$ be an $h$-tiling on a pure $n$-dimensional relative simplicial complex $S$. Then,
$$\sum_{k=0}^{n+1} h_k (\tau) X^k(X+1)^{n+1-k} = X\sum_{k=0}^{n} f_k (S) X^k + \sum_{k=0}^{n-1} c_k (\tau) X^k.$$
In particular, two different $h$-tilings on $S$ have same $h$-vectors iff they have same $c$-vector.
\end{theo}
In Theorem \ref{theohvector}, $f(S)=(f_0(S), \dots , f_{n}(S))$ denotes the {\it face vector} of $S$, whose $j$-th entry is the number of $j$-dimensional remaining faces of $S$, see \cite{CKS}. Moreover, $S$ is said to be {\it pure-dimensional} iff all its facets have same dimension. This result recovers Theorem $4.9$ of \cite{SaWel1} when the tiling $\tau$ involves only basic tiles, see also Proposition $2.3$ of \cite{Stanbook}. By Corollary $4.10$ of \cite{SaWel1}, if it moreover uses a unique closed simplex and tiles a simplicial complex $K$, then $h(\tau)$ coincides with the $h$-vector of $K$ \cite{Stan1,Z}, hence our choice of terminology.

Finally, in the case of closed triangulated manifolds, the classical Dehn-Sommerville relations \cite{Klain} provide another link between critical and $h$-vectors, namely.
\begin{theo}
\label{theopalind}
Let $K$ be an $n$-dimensional triangulated homology manifold equipped with an $h$-tiling $\tau$. Then, the polynomial 
$$\sum_{k=0}^{n+1} h_k (\tau) X^k + \sum_{k=2}^{n+1} c_{n+1-k} (\tau) (X-1)^k + \frac{1}{2} \chi(K)(1-X)^{n+1}$$ 
is palindromic.
\end{theo}
In Theorem \ref{theopalind}, a degree $n$ polynomial $P$ is said to be {\it palindromic} iff it satisfies the identity $X^nP(\frac{1}{X}) = P(X)$ and $\chi(K)$ denotes the Euler characteristic of $K$. The latter satisfies $\chi(K) =  \sum_{k=0}^{n} (-1)^k c_k (\tau)$ for every Morse tiling $\tau$ on $K$, see Lemma $2.5$ of \cite{SaWel2}. When the $h$-tiling uses only basic tiles or when $n\leq3$, Theorem \ref{theopalind} implies that the $h$-polynomial $\sum_{k=0}^{n+1} h_k (\tau) X^k$ is itself palindromic provided $c_0 (\tau) = c_n (\tau)$, as already observed in \cite{Wel1}. We finally  deduce the folllowing corollary.
\begin{cor}
\label{corhvector}
The $h$-vector of any $h$-tiling $\tau$ on an $n$-dimensional triangulated homology manifold satisfies the following identities. 
\begin{enumerate}
\item $\sum_{k=0}^{n+1} k(h_k (\tau) -h_{n+1-k} (\tau)) =0$.
\item $\sum_{k=0}^{n+1} k^2(h_k (\tau) -h_{n+1-k} (\tau)) =0$.
\item $\sum_{k=0}^{n+1} k^3(h_k (\tau) -h_{n+1-k} (\tau)) =6\big((n-1)c_{n-1} (\tau) - 2c_{n-2} (\tau)\big)$ if $n>2$.
\item $\sum_{k=0}^{n+1} k^4(h_k (\tau) -h_{n+1-k} (\tau)) =12(n+1)\big((n-1)c_{n-1} (\tau) - 2c_{n-2} (\tau)\big)$.
\end{enumerate}
\end{cor}
We introduce tilings and discuss their relations with partitionability in section \ref{secrelative}, after having recalled what we need from the theory of simplicial complexes. We prove Theorem \ref{theotiling} and Corollary \ref{cortilingintro} in section \ref{secsubdivision}, together with Theorem \ref{theostellar} in the special case of a single relative simplex and postpone the general case to section \ref{sectilings}. 
We then prove Propositions \ref{propmuK} and \ref{propknotintro} in section \ref{secobstr} and study the stellar complexity $s$ and minimal Morse number $\mu$. We finally introduce critical and $h$-vectors in section \ref{sechvectors} and prove Theorems \ref{theohvector} and \ref{theopalind} together with Corollary \ref{corhvector}.\\

{\bf Acknowledgements:} I am grateful to the referee for his/her valuable report and the many suggestions and references it contained. 

\section{Shellable tilings on relative simplicial complexes}
\label{secrelative}

\subsection{Preliminaries}
\label{subsecprelim}

Let us first recall what we need from the theory of simplicial complexes, see \cite{Koz,Munk}. From the combinatorial point of view, an {\it $n$-simplex} $\sigma$ is a set of cardinality $n+1$ whose elements are called the {\it vertices} of $\sigma$. Any subset of this finite set is called a {\it face}. Its {\it geometric realization} is the convex set
$\vert \sigma \vert = \{ \lambda : \sigma \to \R^+ \, \vert \, \sum_{v \in \sigma} \lambda (v) = 1 \}$, it spans the $n$-dimensional real affine space 
$A_{\sigma} = \{ \lambda : \sigma \to \R \, \vert \, \sum_{v \in \sigma} \lambda (v) = 1 \}$. Likewise, from the combinatorial point of view, a {\it finite simplicial complex} $K$ is a collection of subsets of a finite set $V_K$ which contains all singletons and all subsets of its elements. The elements of $K$ are simplices and any simplex defines itself a finite simplicial complex. The  {\it geometric realization} of a finite simplicial complex $K$ is the subset $\vert K \vert = \{ \lambda : V_K \to \R^+ \, \vert \, \sum_{v \in V_K} \lambda (v) = 1 \text{ and } \text{supp} (\lambda) \in K \}$ of $A_{V_K}$, where $\text{supp} (\lambda) = \{ v \in V_K \, \vert \, \lambda (v) \neq 0 \}$. This topological space is then covered by the geometric realizations of all the simplices of $K$ which are maximal with respect to the inclusion, called {\it facets}, and moreover, any two simplices intersect along a unique common face, possibly empty. A face which has codimension one in any facet containing it is called a  {\it ridge}.The dimension of a simplicial complex $K$  is the maximal dimension of its facets and when they all have same dimension, $K$ is said to be {\it pure dimensional}. It is then said to be {\it strongly connected} iff for any facets $\sigma \neq \sigma'$, there exists a sequence of facets $\sigma = \sigma_0, \sigma_1, \dots , \sigma_N = \sigma'$ such that for every $i \in \{ 1, \dots , N\}$, $\sigma_i$ and $\sigma_{i-1}$ intersect along a ridge. A pure dimensional finite simplicial complex whose ridges are faces of exactly two facets is called a {\it closed pseudo-manifold}. It is said to be a closed {\it triangulated manifold}
when it turns out to be homeomorphic to a closed topological manifold, and more generally a {\it triangulated homology $n$-manifold} when it is homeomorphic to some topological space $M$ whose local homology $H_* (M , M \setminus \{ x \} ; \Z)$ is isomorphic to the relative homology $H_* (\R^n , \R^n \setminus \{ 0 \} ; \Z)$ for every $x \in M$, see \cite{Munk}. These conditions are less restrictive than that of combinatorial or piecewise-linear manifolds which we do not use in this paper, see \cite{RS} and \S \ref{secobstr}. Note that any function from $V_K$ to some real affine space $E$ extends to an affine map $A_{V_K} \to E$ which restricts to $\vert K \vert$ and when this restriction is injective, it embeds $\vert K \vert$ into $E$. For example, the boundary of any convex simplicial polytope of $\R^n$ is the geometric realization of a triangulated sphere, embedded into $\R^n$.

The {\it first barycentric} subdivision $\Sd (K)$ of a finite simplicial complex $K$ is a collection of sets $\{ \sigma_0 , \dots , \sigma_q \}$ of elements of $K$ such that $\emptyset \neq \sigma_0 \subsetneq \sigma_1 \subsetneq \dots \subsetneq \sigma_q$, so that $V_{\Sd (K)} = K \setminus \{ \emptyset \}$. The map $\sigma \in K \setminus \{ \emptyset \} \mapsto \hat{\sigma} \in \vert K \vert \subset A_{\sigma}$, where $\hat{\sigma}$ denotes the barycenter of $\vert \sigma \vert$, defines by extension an homeomorphic embedding $\vert \Sd (K) \vert \to \vert K \vert$, see Proposition $2.33$ of \cite{Koz}. Likewise, the {\it stellar} subdivision $\st_K (\tau)$ of a finite simplicial complex $K$ at a simplex $\tau$ is the collection of subsets of $V_K \cup \{ {\tau} \}$ consisting of the simplices of $K$ that do not contain $\tau$ together with, for every $\sigma \in K$ which contains $\tau$, all cones with apex $\hat{\tau} $ over the faces of $\sigma$ not containing $\tau$. It is thus obtained by first deleting $\tau$ to $K$, that is removing to $K$ all simplices that contain $\tau$, and then by adding the cone with apex $\hat{\tau} $ over the boundary of the star of $\tau$, see \cite{Koz}. The map $V_K \cup \{ {\tau} \} \to \vert K \vert$ which maps $v \in V_K$ to its indicatrix and ${\tau} $ to the barycenter $\hat{\tau} $ of $\vert \tau \vert$ extends to an homeomorphic embedding $\vert \st_K (\tau) \vert \to  \vert K \vert \subset A_{V_K}$. Performing stellar subdivisions at all the non-empty faces of $K$, starting from the top dimensional ones and decreasing the dimension one by one, produces its barycentric subdivision, see for instance Proposition $2.23$ of \cite{Koz}. Finally, a subcomplex $L$ of a finite simplicial complex $K$ is an {\it elementary collapse} of $K$ iff $K \setminus L$ consists of a facet of $K$ together with one of its ridges not contained in any other facet of $K$ It is a {\it collapse} of $K$ if it is can be obtained from $K$ after a finite sequence of elementary collapses, see for instance \cite{AdipBen,For1}. 

\subsection{Relative simplicial complexes and their tilings}
\label{subsectiling}

We now recall the framework of relative simplicial complexes introduced by R. Stanley in \cite{Stan3}, see also \cite{Stanbook,CKS}.

\begin{defi}
\label{defrelativesimplex}
A relative simplex $P$ is a simplex $\overline{P}$ deprived of several of its proper faces $\tau_0, \dots , \tau_k$. A face of $P$ is a relative simplex $\tau \setminus (\tau_0 \cup \dots \cup \tau_k)$, where $\tau$ is a face of its underlying simplex $\overline{P}$ not contained in $\tau_0 \cup \dots \cup \tau_k$, and its dimension is the dimension of $\tau$. 
\end{defi}

The geometric realization of $P$ is the complement $\vert P \vert = \vert \overline{P} \vert \setminus (\vert  \tau_0 \vert  \cup \dots \cup \vert  \tau_k \vert )$, while from the combinatorial point of view, $\tau_0, \dots , \tau_k$ and their faces are no more faces of $P$. Special cases of relative simplices are of particular interest to us. First, simplices that have been deprived only of codimension one faces. There are $n+2$ such relative simplices in dimension $n$ up to isomophism and we denote by $T^n_k$ the one deprived of $k=\ord (T^n_k)$ ridges, so that $T^n_0$ is a closed simplex and $T^n_{n+1}$ an open one. The least dimension of a face of $T^n_k$ is $k-1$ and this face, called its {\it restriction set} $r(T^n_k)$, is unique, isomorphic to an open $(k-1)$-simplex, see Lemma \ref{lemmaface}. The second family of relative simplices of interest to us is then obtained by removing this peculiar $(k-1)$-dimensional face to $T^n_k$. We denote by $C^n_k$ the resulting relative simplex, $k \in \{0 , \dots , n \}$, so that $C^n_0$ is a closed simplex, or a closed simplex deprived of the empty set as a face, see \S \ref{subsecgeompartition}, and $C^n_n = T^n_{n+1}$ an open one, see \cite{SaWel2}. This leads us to the following definition.
\begin{defi}
\label{deftiles}
A basic tile of dimension $n$ and order $k$ is an $n$-simplex deprived of $k$ ridges. A critical tile of dimension $n$ and index $k$ is an $n$-dimensional basic tile of order $k$ deprived of its $(k-1)$-dimensional face.
\end{defi}
The tiles $(C^n_k)_{k \in \{0 , \dots , n \}}$ are thus the critical ones in dimension $n$ while the tiles $T^n_k$, $k \in \{1 , \dots , n \}$, are said to be {\it regular}, see \cite{SaWel1,SaWel2}. A third larger family of relative simplices appears to be of interest to us as well, namely the following ones which have been introduced in \cite{SaWel2}, see Remark \ref{remc32}.
\begin{defi}
\label{defMorsetiles}
A Morse tile of dimension $n$ and order $k$ is an $n$-simplex deprived of $k$ ridges together with possibly a unique face of higher codimension. 
\end{defi}
All Morse tiles which are not given by Definition \ref{deftiles} are regular, a terminology which originates from their connection with discrete Morse theory, see \cite{SaWel2} and Remark \ref{remspectral}.
We now recall the definition of relative simplicial complexes \cite{Stan3,Stanbook}.

\begin{defi}
\label{defrelativesimplicialcomplex}
A relative finite simplicial complex $S$ is a collection of relative simplices $\{ \sigma \setminus (\sigma \cap L) \, \vert \, \sigma \in K \}$, where $L$ is a subcomplex of a finite simplicial complex $K$. \end{defi}
Thus, if $P_1 \subset P \subset P_2$ are relative simplices such that $P_1, P_2 \in S$, then $P \in S$ as well. We may assume the subcomplex $L$ of $K$ not to contain any maximal simplex of $K$, deleting them from $K$ and $L$ otherwise.
We denote by $\overline{S}$ the collection of simplices $\overline{P}$ underlying the elements $P$ of $S$ together with their faces, and call with some abuse this subcomplex of $K$ the underlying simplicial complex of $S$, while $L \cap \overline{S}= \overline{S} \setminus S$ is the collection of faces of $\overline{S}$ not in $S$. The geometric realization of $S$ is the complement $\vert S \vert = \vert \overline{S} \vert \setminus \vert L \vert$. The product of a closed simplex with an open one, once triangulated, provides an example of relative simplicial complex, called a handle in \cite{SaWel2,Wel1}. 

We are mainly interested in the following geometric structure on relative simplicial complexes. 
\begin{defi}
\label{deftiling}
A tiling of a relative simplicial complex $S$ is a partition of its geometric realization by relative simplices such that for every $d \geq 0$, the union of relative simplices of dimensions greater than $d$ is closed in $\vert S \vert$. It is shellable iff it admits a filtration $\emptyset = S_0 \subset S_1 \subset \dots \subset S_N = S$, called a shelling, by closed subsets of $\vert S \vert$ such that for every $p \in \{1 , \dots , N\}$, $S_p \setminus S_{p-1}$ consists of a single relative simplex of the tiling. It is said to be an $h$-tiling (resp. a Morse tiling) iff all the relative simplices involved are given by Definition \ref{deftiles} (resp. Definition \ref{defMorsetiles}).
\end{defi}
The closedness condition in Definition \ref{deftiling} forces the closures of all the relative simplices involved in the tiling to be maximal with respect to the inclusion. 

\subsection{h-tilings as geometric partitionings}
\label{subsecgeompartition}

We now recall the classical notions of shellings and partitionings, see \cite{Stanbook,Z}, and discuss their relations with the $h$-tilings introduced in Definition \ref{deftiling}.
A pure dimensional finite simplicial complex $K$ is classically said to be {\it shellable} whenever its facets $\sigma_0, \dots, \sigma_N$ can be numbered in such a way that for every $i \in \{ 0, \dots , N-1 \}$, $K_i = \sigma_0 \cup \dots \cup \sigma_i$ is connected and $\sigma_{i+1} \setminus \sigma_i$ is a basic tile, see \cite{Stanbook,Z,BjoW}. Such a shelling provides a partitioning of the {\it face poset} of $K$, consisting of its simplices equipped with the partial order given by inclusion, by the closed interval $[\overline{r(T_i)}, \sigma_i]= \{ F \in K \, \vert \, \overline{r(T_i)} \subset F \subset \sigma_i \}$, where $T_0 = \sigma_0$ and $T_i = \sigma_i \setminus K_{i-1}$ if $i>0$. A pure dimensional finite simplicial complex $K$ is more generally said to be {\it partitionable} if and only if its face poset carries such a partitioning by closed intervals, see \cite{KleinSmi,Stanbook,Z}. Any complex which admits a convex embedding in some affine space is partitionable for instance \cite{KleinSmi,Stanbook}, while M. E. Rudin's triangulated tetrahedron is not shellable \cite{Rud}, see also \cite{Z}. Any partitioning uses a unique closed simplex $\sigma$, for the empty face of $K$ has to be contained is a unique closed interval $[\emptyset , \sigma]$. It induces a partitioning of the face poset of the relative simplicial complex $K \setminus \{ \emptyset \}$ by closed intervals together with the semi open one $]\emptyset , \sigma] = [\emptyset , \sigma] \setminus \{ \emptyset \}$. In fact, any open simplex $T$ in such a partitioning, which contributes as a singleton, is itself rather a semi-open interval $]\overline{r(T)}, \overline{T}]$, where $\overline{r(T)}$ denotes a ridge of $\overline{T}$. In the terminology of Definition \ref{deftiles}, these semi-open intervals are critical tiles of minimal and maximal indices, while the relative simplicial complex $K \setminus \{ \emptyset \}$ is a combinatorial counterpart to the geometric realization of $K$. An $h$-tiling in the sense of Definition \ref{deftiling} provides a partitioning of $K \setminus \{ \emptyset \}$ by either closed intervals $[\overline{r(T)}, \overline{T}]$ with ${r(T)} \subsetneq {T}$, the basic tiles, or semi-open intervals $]\overline{r(T)}, \overline{T}]$, the critical tiles, these tiles being in one-to-one correspondence with the facets of $K$. It is closely related to the discrete Morse theory of R. Forman \cite{For1,SaWel2}, see Remark \ref{remspectral}, and its existence is much less restrictive than classical partitionnings by Theorem \ref{theostellar}. 
An intermediate notion, already less restrictive, is the existence of $h$-tilings using only closed or open simplices as critical tiles, for it may use either none or several closed simplices. Figure \ref{figexotic} provides an example of such an $h$-tiling on the boundary of the two-simplex which is not a partitioning in the sense of \cite{Stanbook}. Figure \ref{figcylinder} provides an example of $h$-tiling on the annulus, which once capped with two open simplices provides an $h$-tiling of a triangulated two-sphere without any closed simplex. Depriving the annulus of its boundary and capping it with two closed simplices provides another $h$-tiling of the same triangulated two-sphere without open simplex. 

 \begin{figure}[h]
   \begin{center}
    \includegraphics[scale=1]{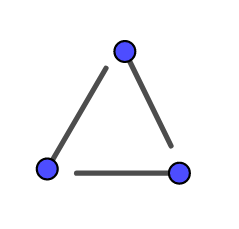}
    \caption{An $h$-tiling of the triangle.}
    \label{figexotic}
      \end{center}
 \end{figure}

\begin{figure}[h]
   \begin{center}
    \includegraphics[scale=1]{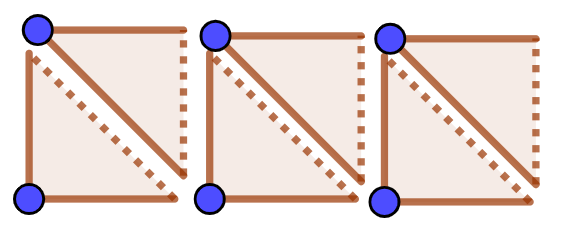}
    \caption{An $h$-tiling of the annulus.}
    \label{figcylinder}
      \end{center}
 \end{figure}

Triangulating the product of several copies of the tiling given by Figure \ref{figexotic} with the shelled boundary of a simplex, we proved in \cite{Wel1} that any product of a sphere with a torus carries an $h$-tiled triangulation using only open and closed simplices as critical tiles, in fact using none of them as soon as the torus has positive dimension. The latter, as the $h$-tiling given by Figure \ref{figexotic}, are not shellable. In fact, every tiling supports a quiver which is acyclic if and only if the tiling is shellable, see Theorem $1.1$ of \cite{Wel2}. Let us finally observe that the simplicial complex built out of two simplices sharing a non-empty common face of codimension greater than one carries no $h$-tiling, see Figure \ref{figstellar}, Example \ref{exsK} and Proposition \ref{propm}.

\section{Subdivisions of a relative simplex}
\label{secsubdivision}

\subsection{Proof of the tiling theorem}
\label{subsectilingtheorem}

Let us first prove Theorem \ref{theotiling}, which  recovers after successive applications the tiling theorems under barycentric subdivisions obtained in \cite{SaWel1,SaWel2}. 

\begin{proof}[Proof of Theorem \ref{theotiling}]
Let us denote by $\sigma = \overline{T}$ the underlying simplex of $T$ and number its ridges $\sigma_0, \dots, \sigma_n$ in such a way that $T=\sigma \setminus  (\sigma_0 \cup \dots \cup \sigma_{k-1} \cup \mu)$, $k$ being the order of $T$. This ordering induces a shelling of $\partial \sigma$ and we set $T_0 = \sigma_0$ and for every $i \in \{1, \dots , n \}$, $T_i = \sigma_i \setminus (\sigma_0 \cup \dots \cup \sigma_{i-1})$ the corresponding basic tiles. We may moreover choose this ordering in such a way that there exist integers $k_1 \leq k_2$ such that $0 \leq k_1 \leq k$, $k-1 \leq k_2 \leq n$ and for every $i \in \{0, \dots , n \}$, $\tau$ is contained in $\sigma_i$ if and only if $i \notin \{ k_1, \dots , k_2\}$, so that $k_2-k_1$ is the dimension of $\tau$. The stellar subdivision of $\sigma$ at $\tau$ is then shelled by the cones with apex $\hat{\tau}$ over $\sigma_{k_1}, \dots , \sigma_{k_2}$, so that $\st_\tau (\sigma \setminus (\sigma_{0} \cup \dots \cup \sigma_{k_1 -1})) = T'_{k_1} \sqcup \dots \sqcup T'_{k_2}$, where for every  $i \in \{0, \dots , n \}$, the join $T'_i = \hat{\tau} * T_i$ is a basic tile of order $i$. We deduce that $\st_\tau (\sigma \setminus (\sigma_{0} \cup \dots \cup \sigma_{k -1}))$ inherits the shelling $(T'_{k_1} \setminus T_{k_1}) \sqcup \dots \sqcup (T'_{k-1} \setminus T_{k-1}) \sqcup \dots \sqcup T'_{k_2}$. This implies the result if $\mu$ is empty, that is if $T$ is a basic tile, for if $k_1 < k$ (resp. $k_1=k$, so that $k<n$ since $\dim (\tau) >0$), the tile $T'_{k-1} \setminus T_{k-1}$ (resp. $T'_{k}$) is isomorphic to $T$ while the other ones are regular basic tiles.

If $\mu \neq \emptyset$ does not contain $\tau$, we may choose our numbering of the ridges in such a way that $\mu  \subset \sigma_k$, since $\mu \not\subset \sigma_{0} \cup \dots \cup \sigma_{k -1}$. In this case, removing $\mu$ from the previous shelling, $\st_\tau (T)$ inherits the Morse shelling $(T'_{k_1} \setminus T_{k_1}) \sqcup \dots \sqcup (T'_{k-1} \setminus T_{k-1}) \sqcup (T'_k \setminus \mu) \sqcup T'_{k+1} \sqcup \dots \sqcup T'_{k_2}$. The Morse tile $T'_k \setminus \mu$ is of order $k$ and thus isomorphic to $T$, while the other ones are regular basic tiles, for $k$ has to be less than $n$ and $k+1$ positive. 
In the special case where $\tau$ strictly contains $\mu$ and $T$ is regular, so that $0=k_1 \leq  l < k_2 \leq n$, we may assume that $\mu$ is contained in $\sigma_j$ iff $j >l = \dim (\mu)$. Then, $\st_\tau (T)$ inherits the Morse shelling $(T'_{0} \setminus T_{0}) \sqcup \dots \sqcup (T'_{k-1} \setminus T_{k-1}) \sqcup (T'_k \setminus (\mu \cap T_k)) \sqcup \dots \sqcup (T'_l \setminus (\mu \cap T_l)) \sqcup (T'_{l+1} \setminus (\mu \cap T_{l+1})) \sqcup T'_{l+2} \sqcup \dots \sqcup T'_{k_2}$. For every $j \in \{ k , \dots , l \}$, $\mu \cap T_l$ is a basic tile of dimension $l-1$ and order $j$, while $\mu \cap T_{l+1}$ is an open $l$-simplex. The tiles $T'_l \setminus (\mu \cap T_l)$ and $T'_{l+1} \setminus (\mu \cap T_{l+1})$ are thus critical of indices $l$ and $l+1$, while the $l-k$ tiles $T'_j \setminus (\mu \cap T_j)$, $j \in \{ k , \dots , l -1 \}$, are regular Morse tiles with $(l-1)$-dimensional Morse faces and orders ranging from $k$ to $l-1$ . 
Finally, if $\mu$ contains $\tau$, then as before $\st_\tau (\mu)$ is shelled by the cones with apex $\hat{\tau}$ over $(\sigma_{k_1} \cap \mu), \dots , (\sigma_{k_2} \cap \mu)$, so that $\st_\tau (\mu \setminus (\sigma_{0} \cup \dots \cup \sigma_{k_1 -1})) = \mu_{k_1} \sqcup \dots \sqcup \mu_{k_2}$, where for every  $i \in \{k_1, \dots , k_2 \}$, $\mu_i$ is a basic tile of order $i$. From what preceds, we deduce that $\st_\tau (T)$ inherits the Morse shelling $\big(T'_{k_1} \setminus (T_{k_1} \cup  \mu_{k_1})\big) \sqcup \dots \sqcup \big(T'_{k-1} \setminus (T_{k-1} \cup  \mu_{k-1})\big) \sqcup (T'_k \setminus \mu_k) \sqcup  \dots \sqcup (T'_{k_2} \setminus \mu_{k_2})$. As before, if $k_1 < k$ (resp. $k_1=k$), the tile $T'_{k-1} \setminus (T_{k-1} \cup \mu_{k-1})$ (resp. $T'_{k} \setminus \mu_{k}$) is isomorphic to $T$, while the other ones are Morse regular. Indeed, for every $i \in \{k_1, \dots , k_2 \}$, $\dim (\mu_i) = 1 + \dim (\mu \cap \sigma_i) \geq \ord (T_i) =  \ord (T'_i)$. Hence the result.
\end{proof}

As a special case of Theorem \ref{theotiling}, we get the following corollary.
\begin{cor}
\label{cortiling}
Let $T$ be a basic or critical tile and $\tau$ any face of its underlying simplex not contained in its Morse face. Then, the stellar subdivision of $T$ at $\tau$ carries a shellable $h$-tiling using a critical tile if and only if $T$ is critical and this tile is then unique of the same index as $T$. 
\end{cor}

\begin{proof}
This follows from the case $\tau \not\subset \mu$ of Theorem \ref{theotiling}. 
\end{proof}

\begin{rem}
\label{remc32}
However, if the Morse face contains $\tau$, Corollary \ref{cortiling} fails to be true and the stellar subdivision of a critical tile of intermediate index may not be $h$-tileable, though it is Morse tileable by Theorem \ref{theotiling}. The codimension of $\tau$ in $\sigma$ has to be greater than one and the first example is the critical tile of dimension three and index two, which is a three-simplex deprived of two ridges and one edge. Its stellar subdivision at the edge is shelled by one basic tile of order one deprived of an edge and one critical tile of index two, but it carries no $h$-tiling. Its barycentric subdivision does not seem to be $h$-tileable as well, as observed in Remark $2.20$ of \cite{SaWel2}.
\end{rem}

\subsection{Proof of Theorem \ref{theostellar} in the case of relative simplices}

\begin{prop}
\label{propbaryc}
The first barycentric subdivision of every relative $n$-simplex $P$ carries a shellable Morse tiling. Moreover, the latter uses a Morse tile of order zero (resp. of order $n+1$) iff $P$ has not been deprived of any ridge (resp. has been deprived of all its ridges) and this tile is then unique.
\end{prop}
By barycentric subdivision of a relative simplex $P = \sigma \setminus (\tau_0 \cup \dots \cup \tau_k)$, we mean the relative simplicial complex $\Sd (P) = \Sd (\sigma) \setminus (\Sd(\tau_0) \cup \dots \cup \Sd(\tau_k))$. 

\proof
Let $P=T^n_k \setminus \tau$ be a relative $n$-simplex which has been deprived of $k$ ridges together with a union of higher codimensional faces $\tau$. By successive applications of Theorem \ref{theotiling}, $\Sd (T^n_k)$ carries a shellable tiling which uses only regular $n$-dimensional basic tiles together with a unique closed (resp. open) $n$-simplex if $k=0$ (resp. $k=n+1$), see Theorem $4.2$ of \cite{SaWel1}. Let us denote these tiles by $T_1 , \dots , T_{(n+1)!}$ following the shelling order, where $T_1$ (resp. $T_{(n+1)!}$) is the closed (resp. open) simplex in case $k=0$ (resp. $k=n+1$). By definition, the simplex $\overline{T}_p$ underlying $T_p$ reads $\{ {\sigma}_0^p, \dots , {\sigma}_n^p \}$, where for every $0 \leq i \leq j \leq n$, $\sigma_j^p$ is a $j$-dimensional face of $\overline{P}= \sigma$ containing $\sigma_i^p$. For every $p \in \{1 , \dots , (n+1)!\}$, such that $T_p$ intersects $\Sd (\tau)$, let us denote by $i_p$ the greatest element in $\{ 0 , \dots , n \}$ such that ${\sigma}_{i_p}^p$ is contained in $\tau$. Then, $\sigma_{i}^p$ is contained in $\tau$ for every $0 \leq i \leq i_p$ and moreover $i_p < n-1$ by assumption. We deduce that ${T}_p \cap \Sd(\tau)$ coincides with the face ${T}_p \cap [{\sigma}_0^p, \dots , {\sigma}_{i_p}^p]$, so that ${T}_p \setminus \Sd(\tau)$ is a Morse tile which can moreover be of order zero (resp. $n+1$) only if $p=1$ (resp. $p=(n+1)!$) and if $T_1$ (resp. $T_{(n+1)!}$) is a closed (resp. open) simplex. The result follows.
\qed

\begin{rem}
Performing another barycentric subdivision, it would be possible to guarantee that the Morse tiling given by Proposition \ref{propbaryc} uses only closed simplices as tiles of order zero.
However, it does not seem possible in general to get rid of the regular Morse tiles not given by Definition \ref{deftiles} using only barycentric subdivisions, as the example of a subdivided critical tile of index two in dimension three shows, see Remark \ref{remc32}. Recall that likewise, by the works of K. A. Adiprasito and B. Benedetti \cite{AdipBen,AdipBen2}, every triangulation of a convex polytope becomes collapsible after one barycentric subdivision and shellable after two.
\end{rem}

\begin{prop}
\label{propMorse}
Every non basic regular Morse tile carries a shellable $h$-tiling after $2^{l-k}$ stellar subdivisions at facets, where $k$ denotes its order and $l\geq k$ the dimension of its Morse face. Moreover, the tiling uses $2^{l+1-k}$ critical tiles of indices ranging from $k$ to $l+1$, the ones of indices $k$ and $l+1$ being unique. The same result holds true using stellar subdivisions at ridges instead, or mixed subdivisions at ridges and facets. 
\end{prop}
For a basic or critical tile, no subdivision is needed at all to get an $h$-tiling, so that we assume $l \geq k$ in Proposition \ref{propMorse}.
\proof
The result follows by induction from Theorem \ref{theotiling}. Let $T$ be an $n$-dimensional Morse tile of order $k$ with $l$-dimensional Morse face $\mu$, $l \geq k$. We perform a first stellar subdivision, either at its facet, or at any ridge containing $\mu$. By the last part of Theorem \ref{theotiling}, the resulting relative simplicial complex inherits a Morse shelling using two critical tiles of indices $l$ and $l+1$ together with $l-k$ regular Morse tiles of orders ranging from $k$ to $l-1$, the remaining tiles being basic and regular. For each of the $l-k$ regular Morse tiles of order $j \in \{ k , \dots, l-1 \}$, we then perform a stellar subdivision either at its facet, or at any ridge containing its Morse face, so that again by the last part of Theorem \ref{theotiling}, the subdivided tile inherits a Morse shelling with two critical tiles of indices $l-1$ and $l$ together with $l-1-j$ regular Morse tiles having $(l-2)$-dimensional Morse faces. Moreover, if the ridge is adjacent to another Morse tile, the latter gets subdivided as well, but inherits a Morse shelling using an isomorphic tile together with regular basic tiles by the first part of Theorem \ref{theotiling}, so that the $l-k$ stellar subdivisions have to be performed one after the other.
This second step of the induction requires thus ${l-k \choose 1}$ stellar subdivisions and produces a Morse tiling with ${l-k \choose 2} = \# \{ k \leq j_1 \leq j_2 \leq l-2 \}$ non basic regular Morse tiles together with  ${l-k \choose 1}$ pairs of critical tiles of indices $l-1$ and $l$, while the $m$-th step of the induction requires ${l-k \choose m-1}$ stellar subdivisions and produces a Morse tiling with ${l-k \choose m} = \# \{ k \leq j_1 \leq \dots \leq j_m \leq l-m \}$ non basic regular Morse tiles together with  ${l-k \choose m-1}$ pairs of critical tiles of indices $l-m+1$ and $l-m+2$. The result follows after $l-k+1$ steps and $\sum_{m=0}^{l-k} {l-k \choose m} = 2^{l-k}$ stellar subdivisions which produce $2 \times 2^{l-k}$ critical tiles of indices ranging from $k$ to $l+1$, where however only the first (resp. last) subdivision produces a critical tile of index $l+1$ (resp. $k$).
\qed

\begin{ex}
\label{exstellar}
If $P$ is a three-simplex deprived of one ridge and one edge, then $\st_{\overline{P}} (P)$ gets tiled by two critical tiles of consecutive indices one and two and two basic tiles of order one and three.
\end{ex}
Let us now prove Corollary \ref{cortilingintro}. 

\begin{proof}[Proof of Corollary \ref{cortilingintro}]
The first part of Corollary \ref{cortilingintro} follows from Corollary \ref{cortiling} and concatenation of the shelling orders. The remaining parts of Corollary \ref{cortilingintro} follow from Theorem \ref{theotiling} along the same lines as Proposition \ref{propMorse}.
\end{proof}

In fact, there is no need of barycentric subdivision  to get a Morse tileable subdivision as in Proposition \ref{propbaryc}. 

\begin{prop}
\label{proprelative}
Every relative simplex $P$ carries a shellable Morse tiling after finitely many stellar subdivisions at facets, or also after finitely many stellar subdivisions at ridges, or after mixed ones. Moreover, the tiling uses a Morse tile of order zero (resp. an open simplex) iff $P$ has not been deprived of any ridge (resp. has been deprived of all its ridges) and this tile is then unique.
\end{prop}

\proof
Let $P=\sigma \setminus (\tau_0 \cup \dots \cup \tau_{k})$ be a relative $n$-simplex, where $\tau_0, \dots, \tau_{k}$ are non-empty proper faces of the $n$-simplex $\sigma$. We may assume that
$\tau_0, \dots , \tau_{k_1 - 1}$ have codimension one in $\sigma$ and $\tau_{k_1}, \dots , \tau_{k}$ codimensions greater than one, with $0 \leq k_1 \leq k-1$. We then denote by $v_0, \dots, v_n$ the vertices of $\sigma$ in such a way that for every $0 \leq j \leq k_1$, $v_j$ is not contained in $\tau_j$ and that $v_{k_1}$ is a vertex of $\tau_{k_1 + 1}$. If we denote by $\sigma_j $ the ridge of $\sigma$ opposite to $v_j$, $0 \leq j \leq n$, this labelling induces a shelling $T_0 \sqcup \dots \sqcup T_n$ of $\partial \sigma$, where $T_0 = \sigma_0$ and for every $j \in \{ 1 , \dots , n \}$, $T_j = \sigma_j \setminus (\sigma_0 \cup \dots \cup \sigma_{j-1})$. Then, $T_0 \cup \dots \cup T_{k_1-1}$ coincides with
$\tau_0 \cup \dots \cup \tau_{k_1-1}$, $T_ {k_1} \cap \tau_{k_1+1}$ has smaller dimension than $ \tau_{k_1+1}$  and $T_j \cap  \tau_{k_1} = \emptyset$  for every $j > k_1$. The stellar subdivision of $\sigma$ at its maximal face is then shelled by $\widetilde{T}_0 \sqcup \dots \sqcup \widetilde{T}_n$, where for every $j \in \{ 0 , \dots , n \}$, $\widetilde{T}_j$ is the cone with apex $\hat{\sigma}$ over $T_j$. It is a basic tile of order $j$ which contains $\hat{\sigma}$ only if $j=0$, see Proposition $4.1$ of \cite{SaWel1}. The stellar subdivision of $P$ at $\sigma$ is then shelled by the relative simplices $T'_0 \sqcup \dots \sqcup T'_n$, where for every $j \in \{ 0 , \dots , n \}$, $T'_j = \widetilde{T}_j \setminus (\tau_0 \cup \dots \cup \tau_{k})$. In particular, for every $j \in \{ 0 , \dots , k_1-1 \}$, $T'_j$ is a basic tile of order $j+1$, see Proposition $4.1$ of \cite{SaWel1}, for every $j \in \{ k_1+1 , \dots , n \}$, $T'_j$ has been deprived of less faces of codimensions greater than one than $P$ and the dimension of $\overline{T}'_{k_1} \cap \tau_{k_1 +1}$ is less than the one of  $ \tau_{k_1 +1}$. 
Likewise, the stellar subdivision of $\sigma$ (resp. $\sigma \setminus \sigma_0$) at $\sigma_n$ (resp. $\sigma_0$) is shelled by the cones with apex $\hat{\sigma}_n$ (resp.  $\hat{\sigma}_0$ ) over 
$T_0 \cup \dots \cup T_{n-1}$ (resp. $T_1 \cup \dots \cup T_{n}$). The stellar subdivision of $P$ at $\sigma_n$ (resp. at $\sigma_0$ if $k_1 >0$) is then shelled by the latter deprived of $\tau_0 \cup \dots \cup \tau_{k}$ (resp. $\tau_1 \cup \dots \cup \tau_{k}$) and these relative simplices either have been deprived of less faces of codimensions greater than one than $P$, or of faces of lower dimensions. In all cases, we deduce the result after finite induction, since in the case of stellar subdivision at a ridge, we can always assume this ridge to be either $\sigma_0$ or $\sigma_n$ depending on whether  it belongs to $P$ or not, and at each step of the induction, choose this ridge adjacent to a relative simplex having more than one higher codimensional missing face and for which the total dimension of the latter is maximal among all the relative simplices of the tiling. The resulting shelled tiling cannot use any open simplex unless $P$ is itself an open one while it uses a Morse tile of order zero, which is unique, iff $k_1=0$.
\qed

\begin{rem}
\label{remalgo}
1) The proof of Proposition \ref{proprelative} relies on the fact that given any ridge $\sigma$ of a relative simplex $P$, $\st_\sigma (P)$ inherits a shellable tiling by relative simplices for which the total dimension of the higher codimensional missing faces has decreased, provided they have been deprived of at least two such faces. 

2) The proofs of Propositions \ref{propbaryc}, \ref{propMorse} and \ref{proprelative} are algorithmic. Moreover, in bounded dimension the number of relative simplices is finite so that these algorithms produce the sequence of subdivisions together with the shellings in finite bounded time.
\end{rem}

\section{Existence of shellable $h$-tilings}
\label{sectilings}

We are now ready to prove Theorem \ref{theostellar} which provides the existence of shellable $h$-tilings on all finite relative simplicial complexes after finitely many stellar subdivisions at facets, or ridges. 

\proof[Proof of Theorem \ref{theostellar}]
Let $S = \overline{S} \setminus L$ be a relative simplicial complex, where $L$ denotes a subcomplex of the finite simplicial complex $\overline{S}$ which does not contain any maximal simplex. Let us number the facets of $\overline{S}$ in decreasing dimensions by $\sigma_1, \dots, \sigma_N$. It induces a filtration $\emptyset = \overline{S}_0 \subset \overline{S}_1 \subset \dots \subset \overline{S}_N = \overline{S}$ of subcomplexes, where for every $j \in \{ 1, \dots , N \}$, $\overline{S}_j$ denotes the complex containing $\sigma_1, \dots, \sigma_j$ together with their faces. We then set $P_j = \overline{S}_j \setminus (\overline{S}_{j-1} \cup L)$. By construction, for every $d \geq 0$, the union of these relative simplices which have dimensions greater than $d$ is closed in $\vert S \vert$, for there exists $j_d \in  \{ 1, \dots , N \}$ such that it coincides with $\vert \overline{S}_{j_d} \setminus L \vert$. Proposition \ref{propbaryc} then provides a Morse shelling on $\Sd (P_j)$ for every $j \in \{ 1, \dots , N \}$ and the first part of Theorem \ref{theostellar} follows by concatenation of these shelling orders. Likewise, Propositions \ref{proprelative} and \ref{propMorse} provide a finite sequence of stellar subdivisions at facets on each relative simplex $P_j$, $j \in \{ 1, \dots , N \}$, together with a shelled tiling on the resulting subdivided relative simplex. The second part of Theorem \ref{theostellar} again follows by concatenation of these shelling orders in the case of stellar subdivisions at facets. 
In order to get the result using stellar subdivisions at ridges instead, we proceed by induction as in the proof of Proposition \ref{proprelative} to turn this shelling by relative simplices into some Morse shelling. At each step of the induction, we choose a (non missing) ridge of a relative simplex $P$ having more than one higher codimensional missing faces, the total dimension of the latter being maximal among all the relative simplices of the tiling. By the closedness condition of Definition \ref{deftiling}, this ridge is a ridge of $S$ as well. 
Then, as observed in the proof of Proposition \ref{proprelative}, the subdivided relative complex inherits a shelling by relative simplices whose higher codimensional missing faces, when they are at least two of them, either have lower total dimension than the ones of $P$, or have same dimension, but the number of such tiles has decreased, see Remark \ref{remalgo}. After finitely many steps, we thus get a Morse shelled finite relative simplicial complex.  Now, in order to turn this Morse shelling into some shelled $h$-tiling, we again proceed by finite induction as in the proof of Proposition \ref{propMorse}. At each step of the induction, we choose a non basic regular Morse tile of the tiling together with a ridge containing its Morse face and perform a stellar subdivision at this ridge which by the closedness condition of Definition \ref{deftiling} is a ridge of $S$ as well. The subdivided tile itself inherits a Morse shelling whose non basic regular Morse tiles have lower dimensional Morse faces by the last part of Theorem \ref{theotiling} while each of the other tiles adjacent to this ridge, once subdivided, inherits a Morse shelling using one isomorphic tile together with regular basic tiles by the first part of Theorem \ref{theotiling}. We thus get the result after finite induction. 
Finally, in bounded dimension, the sequences of subdivisions together with the shellings given by Propositions \ref{propbaryc}, \ref{proprelative} and \ref{propMorse} are produced in bounded finite time by the corresponding algorithms, see Remark \ref{remalgo}. The time complexity of the algorithm producing the subdivisions and shellings of $S$ is thus of the same order as the complexity of the algorithm which ranges the facets of $S$ in decreasing dimensions and then, for each facet, computes its intersection with the previous ones. It is thus quadratic in the size of $S$, given by its number of facets. 
\qed \\

In the case of closed pseudo-manifolds, Theorem \ref{theostellar} can be slightly precised and the algorithm given in its proof slightly improved. 
\begin{theo}
\label{theostellarman}
In the case of closed strongly connected pseudo-manifolds, the shellable tilings given by Theorem \ref{theostellar}, either after one barycentric subdivision or after finitely many stellar subdivisions, can be chosen to use a unique closed simplex and at least one open one,  and no other tiles of order zero. 
\end{theo}
Theorem \ref{theostellarman} gets much stronger than Theorem $1.2$ of \cite{Wel1} which guarantees the existence of Morse shellable triangulations on all finite products of closed manifolds of dimensions less than four. In the case of closed three-manifolds, we proved in \cite{SaWel2} that such a Morse shelling can be chosen to use the same number of critical tiles, with same indices, as any given smooth Morse function on the manifold, compare Proposition \ref{propknot} and \cite{Ben}. 

\proof
We first proceed as in the proof of Theorem $1.3$ of \cite{SaWel2}, which concerns the case of surfaces, in which case no subdivision is needed at all. Let $K$ be a closed strongly connected pseudo-manifold and let $P_1$ be any of its facets. We set $K_1 = P_1$ and proceed by finite induction. As long as $K_j$, $j \geq 1$, is not a pseudo-manifold, it contains a ridge which belongs to a single facet $\theta$ of $K_j$, while by assumption it is contained in two facets $\theta$ and $\sigma_{j+1}$ of $K$.
We then set $K_{j+1} = K_j \cup \sigma_{j+1}$ together with their faces and $P_{j+1} = K_{j+1} \setminus K_j$. We get after finite induction a filtration
$\emptyset = K_0 \subset K_1 \subset \dots \subset K_N $ of subcomplexes, such that $K_N$ is a pseudo-manifold, so that $K_N = K$ by strong connectedness, see \S \ref{subsecprelim}. Moreover, for every $j \geq 1$, $K_j \setminus K_{j-1}$ is the relative simplex $P_j$ which has been deprived of at least one ridge as soon as $j>1$, while $P_N$ does not contain any ridge, so that it is an open simplex. If $n \leq 2$, all these relative simplices are tiles given by Definition \ref{deftiles} and we get a shelled tiling of $K$ using exactly one closed simplex and at least one open one. 
Otherwise, we apply to the filtration $K_0 \subset K_1 \subset \dots \subset K_N $ the algorithm given in the proof of Theorem \ref{theostellar} to get the result.  
\qed

\begin{rem}
\label{remspectral}
1) Performing a single barycentric subdivision on a simplicial complex $K$ requires however to subdivide all its simplices, whereas to get the last parts of Theorems \ref{theostellar} or  \ref{theostellarman}, the algorithm subdivides only some facets or ridges of $K$, which might well be a small amount of them. 

2) Any shelling on a relative simplicial complex $S = K \setminus L$ given by Theorems \ref{theostellar} or  \ref{theostellarman} provides two spectral sequences which converge to the relative (co)homology of the pair $(K,L)$ and whose first pages are spanned by the critical tiles of the tilings, see \cite{Wel2}. They provide a way to compute this (co)homology which is alternative to the discrete Morse theory of R. Forman \cite{For1,For2}. These shellings also encode a class of discrete Morse functions whose critical points are in one-to-one correspondence with the critical tiles of the shellings, preserving the index, since a critical tile of index $k$ collapses on any of its $k$-face while a regular Morse tile consists of a sequence of collapses, see \cite{SaWel2}. 
\end{rem}

\section{Obstructions to shellability}
\label{secobstr}

The algorithm given in the proof of Theorem \ref{theostellar}, which lies in the complexity class $P$, may produce more stellar subdivisions than necessary to tile or shell the complex. 

\begin{defi}
\label{defsK}
For every finite relative simplicial complex $S$, let $s'(S)$ (resp. $s(S)$) be the minimal number of stellar subdivisions at facets required for it to be $h$-tileable (resp. to carry a shellable $h$-tiling).
\end{defi}

By Theorem \ref{theostellar}, $0 \leq s'(S) \leq s(S) < +\infty$ and the {\it stellar complexity} $s(S)$ vanishes if and only if $S$ carries a shellable $h$-tiling. Recall that deciding whether a finite simplicial complex is shellable in the classical sense, or collapsible (resp. contractible), is $NP$-complete (resp.  undecidable) by \cite{Goaoc,Tan} (resp. \cite{VKF}, see also \cite{Tan}). What about the complexity of deciding whether its stellar complexity vanishes or of computing it?

\begin{ex}
\label{exsK}
The simplicial complex $K$ consisting of two triangles sharing a common vertex is not $h$-tileable, as already observed in \S \ref{subsecgeompartition}. It however carries a shellable $h$-tiling after a single stellar subdivision at one triangle, see Figure \ref{figstellar} where the numbers refer to the indices of the critical tiles used by the shelling, so that $s(K) = s'(K) = 1$. 
\end{ex}

\begin{figure}[h]
   \begin{center}
    \includegraphics[scale=1]{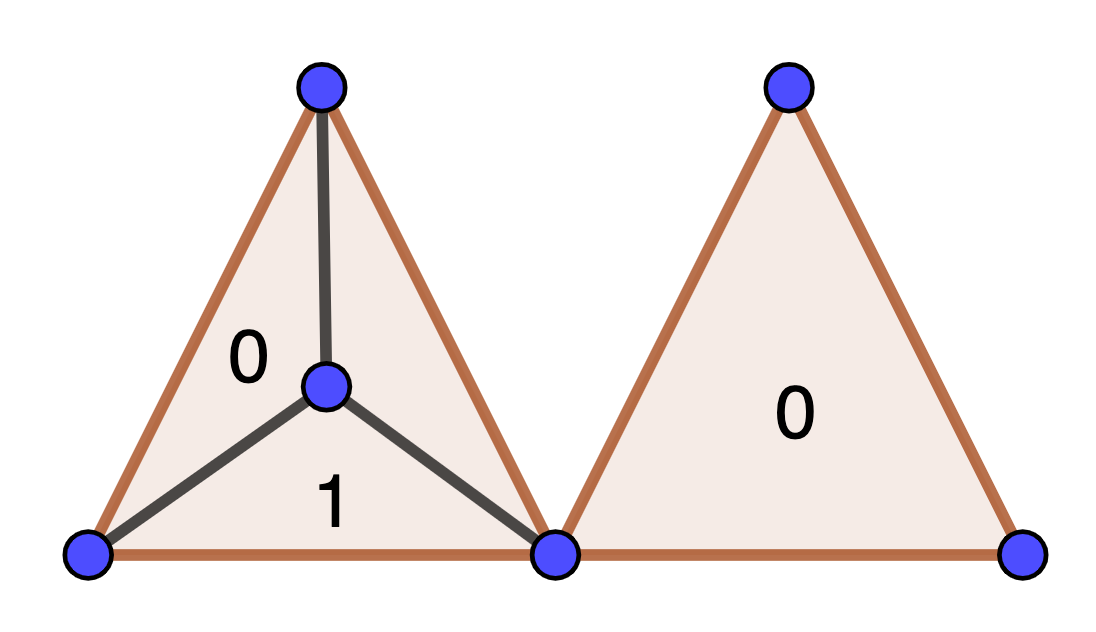}
    \caption{A shelled $h$-tiling.}
    \label{figstellar}
      \end{center}
 \end{figure}
 
 \begin{rem}
In Example \ref{exsK}, whatever the number of stellar subdivisions at facets performed on $K$, every shelled $h$-tiling on the resulting complex requires at least two critical tiles of index zero, see Proposition \ref{propm}, so that the filtration given by the shelling has to contain some disconnected members. This contrasts with the classical definition of shelling which requires each member to be connected and the complex is thus not shellable, though it is collapsible. 
 \end{rem}

If now the finite relative simplicial complex $S$ is $h$-tileable, it may carry many different $h$-tilings with even different critical vectors, in the same way as a finite simplicial complex carries many different discrete Morse functions with different Morse numbers, see \cite{For1}. 

\begin{defi}
\label{defmuK}
For every finite relative simplicial complex $S$ such that $s'(S)=0$ (resp. $s(S)=0$), let $\mu' (S)$ (resp. $\mu(S)$) be the infimum over all $h$-tilings (resp. shellable $h$-tilings) $\tau$ on $S$ of the number of critical tiles used by $\tau$.
\end{defi}

Let us now prove Proposition \ref{propmuK}, up to which $\mu-2$ is the obstruction to shellability in the classical sense for closed pseudo-manifolds of positive dimensions with vanishing stellar complexity. 

\begin{proof}[Proof of Proposition \ref{propmuK}]
Any shelled $h$-tiling of a closed pseudo-manifold $K$ of positive dimension $n$ starts with a critical tile of index zero and ends with a critical tile of index $n$. Indeed, the latter is the only kind of tiles which does not contain any codimension one face while by hypothesis, any ridge of $K$ is contained in two facets. Thus, $\mu (K) \geq 2$, with equality if and only if $K$ admits a shelled $h$-tiling without any other critical tile. Such a shelling is then a shelling in the classical sense. Conversely, any classical shelling of a closed $n$-dimensional pseudo-manifold $K$ starts with a closed simplex and ends with an open one, while the intermediate tiles are regular. It thus defines a shelled $h$-tiling using only two critical tiles, so that $\mu (K) = 2$. 
\end{proof}

By Corollary \ref{cortiling}, the Morse numbers given by Definition \ref{defmuK} can only decrease under stellar subdivisions at facets and thus have to stabilize. 
Given a finite relative simplicial complex $S= \st^0 (S)$, let us denote by $\st (S) = \st^1 (S)$ the relative simplicial complex obtained from $S$ after stellar subdivision at all of its facets and for every $d>0$, set $\st^{d+1} (S) = \st (\st^d (S))$. By Theorems \ref{theostellar} and \ref{theotiling}, these carry shellable $h$-tilings for $d$ large enough.

\begin{defi}
\label{defmuinftyK}
For every finite relative simplicial complex $S$, let $\mu'_\infty (S)$ (resp. $\mu_\infty (S)$) be the infimum over all $d \geq 0$ of $\mu' (\st^d (S))$ (resp. $\mu (\st^d (S))$).
\end{defi}

\begin{lemma}
\label{lemmamu}
For every finite relative simplicial complex $S$, $\mu_\infty (S) = \chi (S) \mod(2)$. Likewise, if $s' (S) = 0$ (resp. $s(S) = 0$), $\mu' (S) = \chi (S) \mod(2)$ (resp. $\mu (S) = \chi (S) \mod(2)$ ).
\end{lemma}
In Lemma \ref{lemmamu}, $ \chi (S) $ denotes the Euler characteristic of $S$, so that if $S = K \setminus L$, $ \chi (S) = \chi (K)- \chi (L)$.
\begin{proof}
By additivity of the Euler characteristic and Lemma $2.5$ of \cite{SaWel2}, for every $h$-tiling $\tau$ on a finite  relative simplicial complex $S$, $ \chi (S) = \sum_{k=0}^{+ \infty} (-1)^k c_k (\tau)$. 
Thus, $ \chi (S) $ coincides modulo two with the total number of critical tiles of $\tau$. The result follows, since the Euler characteristic remains invariant under subdivision. 
\end{proof}  

By Theorem \ref{theotiling}, for every finite simplicial complex $K$ such that $s (K) = 0$, $\mu_\infty (K) \leq \mu (K)$, and by Theorem $1.2$ of \cite{SaWel2}, $\mu (K)$ bounds from above the minimal number of critical points of a discrete Morse function on $K$, see Remark \ref{remspectral}. However, even the asymptotic Morse number $\mu_\infty (K)$ might be strictly larger. 
\begin{prop}
\label{propknot}
Let $K$ be a locally constructible triangulated three-sphere which contains a knotted triangle in its one-skeleton. Then, for every $d \geq 0$, there exists a discrete Morse function on $\st^d (K)$ having only two critical points, while $\mu_\infty (K) \geq 4$.
\end{prop}
Examples of such locally constructible triangulated three-spheres are given in \cite{BenZ}, Examples $2.26$, $2.27$ and $2.28$ and we refer to this paper of B. Benedetti and G. M. Ziegler
for the definition of local constructibility as well. 

\begin{proof}
By Corollary $2.11$ of  \cite{BenZ}, the locally constructible three-sphere $K$ is collapsible once deprived of any facet $\sigma$, so that there exists a sequence of subcomplexes $K_0 \subset K_1 \subset \dots \subset K_N \subset K$ such that $K_N = K \setminus \sigma$, $K_0$ is a single vertex and for every $i \in \{ 0, \dots , N-1 \}$, $K_{i+1}$ collapses on $K_i$ by some elementary collapse. Let $f$ be the discrete function which vanishes on $K_0$, takes the value $N+1$ on $\sigma$ and which for every $j \in \{ 1, \dots , N \}$, takes the value $j$ on both faces of $K_i \setminus K_{i-1}$. Then, $f$ is a discrete Morse function with only two critical faces, namely $\sigma$ and the vertex $K_0$, see \cite{For1}. Now, performing any stellar subdivision does not affect the property of being collapsible once deprived of any facet, see for instance Theorem $10.14$ of \cite{Koz2}, so that all the complexes $\st^d (K)$, $d \geq 0$, carry Morse functions with only two critical points. 
Likewise, performing stellar subdivision at any facet does not affect the property of containing a knotted triangle, that is a non-trivial knot on three edges, in the one skeleton, so that all the complexes $\st^d (K)$, $d \geq 0$, share this property. However, by Theorem $1$ of \cite{HachZ}, this prevents $\st^d (K)$ from being constructible, hence shellable, compare \cite{Frankl,VanK,Lick}. By Proposition \ref{propmuK}, this forces $\mu (\st^d (K)) >2$ as soon as $s(\st^d (K)) = 0$, so that taking the infimum over all $d \geq 0$, we get $\mu_\infty (K) >2$ and thus by Lemma \ref{lemmamu}, 
$\mu_\infty (K) \geq 4$.
\end{proof}

\begin{rem}
\label{rempl}
1) It would be of interest to bound $\mu (K)$ or $\mu_\infty (K)$ from below by some topological invariant of the knot. 

2) By the work of W. B. R. Lickorish \cite{Lick} (see also  \cite{BenZ}), there also exist non locally constructible triangulated three-spheres containing knotted triangles in their one-skeleton, so that the minimal number of critical points of their discrete Morse functions is then greater than two as well. 
\end{rem}

The spheres $K$ given by Proposition \ref{propknot} have the property that none of the subdivisions $\st^d (K)$, $d \geq 0$, are shellable by Proposition  \ref{propmuK}. This contrasts with Theorem $B$ of \cite{AdipBen2} up to which they become shellable after a large number of barycentric subdivisions. These by the way are already Morse shellable after a single barycentric subdivision by Theorem \ref{theostellar}, so that the minimal number of critical tiles used by such a Morse shelling decreases after successive barycentric subdivisions until it reaches two when the complex becomes shellable in the classical sense. The latter may happen after some arbitrary large number of barycentric subdivisions depending on the complexity of the knot by \cite{Lick}. 

\begin{defi}
\label{defm}
For every Morse shellable finite relative simplicial complex $S$, let $m' (S)$ (resp. $m(S)$) be the infimum over all Morse tiling (resp. Morse shelling) $\tau$ on $S$ of the number of critical tiles used by $\tau$. Likewise, for every finite relative simplicial complex $S$, let $m' _\infty(S)$ (resp. $m_\infty (S)$) be the infimum of $m' (\Sd^d (S))$ (resp. $m (\Sd^d (S))$) over all $d \geq 1$. 
\end{defi}
\begin{lemma}
\label{lemmam}
For every finite relative simplicial complex $S$, $m_\infty (S) = \chi (S) \mod(2)$. Likewise, if $S$ is Morse tileable (resp. Morse shellable), $m' (S) = \chi (S) \mod(2)$ (resp. $m (S) = \chi (S) \mod(2)$ ).
\end{lemma}
\begin{proof}
The proof goes along the same lines as the proof of Lemma \ref{lemmamu}.
\end{proof}

By Theorem $B$ of \cite{AdipBen2}, $m_\infty (K)=2$ for any piecewise-linear triangulated sphere $K$. If $K$ is not piecewise-linear, then Theorem \ref{theostellar} still applies and provide a Morse shelling after a single barycentric subdivision and shellable $h$-tiling after finitely many stellar subdivisions at facets or ridges, but no subdivision can be shellable since they remain not piecewise-linear. Such a non-PL sphere may carry some discrete Morse function having only two critical faces by Corollary $2.38$ of \cite{Ben}, what about $m_\infty (K)$?
In general, for every finite relative simplicial complex $S$, how does $m_\infty (S)$ compare with the minimal number of critical points of a discrete Morse function on $\Sd^d (S)$, $d \gg 0$? The former bounds from above the latter, but can these numbers be equal  in general as in the case of piecewise-linear spheres? 
One cannot replace Morse shellings by shelled $h$-tilings, as Example \ref{exsK} shows. 

\begin{prop}
\label{propm}
Let $K$ be the union of two simplices sharing a common vertex. Then, $m(K) =  m_\infty (K) = 1$, while for every $d>0$, $s(\Sd^d (K))=0$ and  $\mu (\Sd^d (K)) =3$. In fact, every shelled $h$-tiling on $\Sd^d (K)$ contains two closed simplices.
\end{prop}
\begin{proof}
The simplicial complex $K$ carries a Morse shelling using a critical tile of vanishing index together with a regular Morse tile of order zero, so that by Definition \ref{defm} and Theorem \ref{theotiling}, $m(K) =  m_\infty (K) = 1$. Also, by Theorem \ref{theotiling}, $\st (K)$ carries a shelled $h$-tiling using two critical tiles of vanishing index and one critical tile of index one, so that after stellar subdivisions at all the remaining faces of $K$, we deduce that $\Sd (K)$ carries such a shelled $h$-tiling as well. Then, by Theorem \ref{theotiling}, all the barycentric subdivisions $\Sd^d (K)$, $d \geq 1$, carry a shelled $h$-tiling using two critical tiles of vanishing index and one critical tile of index one, the remaining tiles being regular. Conversely, for every $d \geq 1$ and every shelled $h$-tiling $\tau$ on $\Sd^d (K)$, denoting by $\sigma_1$ and $\sigma_2$ the two facets of $K$, we observe that the first tile of $\tau$, with respect to the shelling order, which belongs to $\Sd^d (\sigma_i)$, $i \in \{ 1,2 \}$, has to be a closed simplex, since it has to contain all its ridges. Therefore, $\tau$ contains at least two critical tiles of index zero and thus at least one critical tile of index one since $\Sd^d (K)$ is collapsible, so that $\mu (\Sd^d (K)) =3$. Hence the result. 
\end{proof}

\section{Critical and $h$-vectors}
\label{sechvectors} 

Let us recall that the {\it $h$-vector} of a Morse tiling $\tau$ on an $n$-dimensional relative simplicial complex is the vector $h(\tau)=(h_0(\tau), \dots , h_{n+1}(\tau))$ whose $j$-th entry is the number of tiles of order $j$ used by $\tau$, $j \in \{ 0 , \dots , n+1 \}$. Likewise, its {\it critical} or {\it $c$-vector} is the vector $c(\tau)=(c_0(\tau), \dots , c_{n}(\tau))$ whose $j$-th entry is the number of critical tiles of index $j$ used by $\tau$, $j \in \{ 0 , \dots , n \}$. A critical tile of index $j<n$ is a special tile of order $j$ also counted by $h_j (\tau)$, so that $c_j (\tau) \leq h_j (\tau)$ in this case, while $c_n (\tau) = h_{n+1} (\tau)$. Let us also recall the face numbers of basic tiles.

\begin{lemma}[Proposition $4.3$ of  \cite{SaWel1}]
\label{lemmaface}
For every $n>0$ and every  $k\in \{0,\ldots, n+1\}$, 
$$f_j(T_k^n)=\begin{cases} 
0 & \mbox{if } 0\leq j<k-1,\\
\binom{n+1-k}{n-j}& \mbox{if } k-1\leq j\leq n.
\end{cases}$$ 
\end{lemma}

\begin{proof}
When the order $k$ of the $n$-dimensional basic tile equals $0$ or $n+1$, the tile is a closed or open simplex and one checks the result. Otherwise, the tile is isomorphic to a cone deprived of its apex over an $(n-1)$-dimensional basic tile of the same order, so that one gets the result by induction. 
\end{proof}

We are now ready to prove Theorems \ref{theohvector}, \ref{theopalind} and Corollary \ref{corhvector}.

\proof[Proof  of Theorem \ref{theohvector}]
We proceed as in the proof of Theorem $4.9$ of \cite{SaWel1} and compute the face vector of $S$ using the $h$-tiling $\tau$, see also Proposition $2.3$ of \cite{Stanbook}. The face vector of a basic tile satisfies, for every $0 \leq k \leq n+1$, $f_j (T^n_k) = 0$ if $j < k-1$, while $f_j (T^n_k) = {n+1 - k \choose n-j}$ if $k-1 \leq j \leq n$ by Lemma \ref{lemmaface}. By definition then, $f_j (C^n_k) = f_j (T^n_k)$ if $j \neq k-1$ and $f_{k-1} (C^n_k) = 0$. Let us set $\tilde{c}_j (\tau) = c_j (\tau)$ if $j < n$ and $\tilde{c}_j (\tau) = 0$ otherwise. Since $S$ is pure $n$-dimensional, all the tiles of $\tau$ have same dimension $n$ and summing over all of them we get
$$\sum_{j=0}^{n+1} f_{j-1} (S) X^j = \sum_{j=0}^{n+1} X^j \big( \sum_{k=0}^{j} {n+1 - k \choose j-k} h_k (\tau) - \tilde{c}_j (\tau) \big),$$
where we set $f_{-1} (S) = 0$ and where the term $\tilde{c}_j (\tau)$ corrects the fact that $h_j (\tau)$ counts one $(j-1)$-dimensional open face also for each critical tile of index $j<n$. We deduce as in \cite{SaWel1}. 
$$X\sum_{j=0}^{n} f_j (S) X^j + \sum_{j=0}^{n-1} c_j (\tau) X^j = \sum_{k=0}^{n+1} h_k (\tau) X^k(X+1)^{n+1-k}.$$
This identity implies that two $h$-tilings on $S$ have same $h$-vector iff they have same number of critical tiles in each index between $0$ and $n-1$. Now, if $S = \overline{S} \setminus L$ for some subcomplex $L$ of $ \overline{S} $ and if we set $\chi (S) = \chi ( \overline{S} ) - \chi (L) = \chi ( \overline{S} ,L)$, then we know by Lemma $2.5$ of \cite{SaWel2} that this Euler characterisctic can be computed as $\chi (S) = \sum_{k=0}^{n} (-1)^k c_k (\tau)$, so that the number $c_n (\tau)$ is determined by the numbers $c_0 (\tau), \dots , c_{n-1} (\tau)$. Hence the result. 
\qed \\

\proof[Proof  of Theorem \ref{theopalind}]
We proceed as in the proof of Theorem $3.2$ of \cite{Wel1}. By Theorem \ref{theohvector},
$$X\sum_{k=0}^{n} f_k (K) X^k = \sum_{k=0}^{n+1} h_k (\tau) X^k(X+1)^{n+1-k} - \sum_{k=0}^{n-1} c_k (\tau) X^k.$$
Now, by Theorem $2.1$ of \cite{Mac}, the Dehn-Sommerville relations can be expressed as the identity $R_K (-1-X) = (-1)^{n+1}  R_K (X)$, where $R_K (X) = X\sum_{k=0}^{n} f_k (K) X^k - \chi (K) X$, see also Theorem $1.1$ of \cite{SaWelDCG}. We thus deduce
$$\sum_{k=0}^{n+1} h_k (\tau) X^k(X+1)^{n+1-k} - \sum_{k=0}^{n-1} c_k (\tau) X^k - \chi (K) X $$
$$= \sum_{k=0}^{n+1} h_k (\tau) (X+1)^kX^{n+1-k} - \sum_{k=0}^{n-1} c_k (\tau) (-1)^{n+1-k} (X+1)^k + (-1)^{n+1} \chi (K) (X+1).$$
We now set $X = \frac{T}{1-T}$ and observe that  $\chi (K) = 0$ when $n$ is odd by Poincar\'e duality, see \cite{Munk}, to deduce
$$\sum_{k=0}^{n+1} h_k (\tau) T^k - \sum_{k=0}^{n-1} c_k (\tau) T^k(1-T)^{n+1-k} $$
$$= \sum_{k=0}^{n+1} h_k (\tau) T^{n+1-k} - \sum_{k=0}^{n-1} c_k (\tau) (-1)^{n+1-k} (1-T)^{n+1-k} - \chi (K) (1-T)^{n+1}.$$
The result follows, since $\chi (K) = 0$ when $n$ is odd.
\qed \\

\proof[Proof  of Corollary \ref{corhvector}]
Theorem \ref{theopalind} implies that 
$$\sum_{k=0}^{n+1} \big(h_k (\tau) -  h_{n+1-k} (\tau)\big) X^k = \sum_{k=2}^{n+1} c_{n+1-k} (\tau) (1-X)^k \big(X^{n+1-k} - (-1)^k \big) - \chi (K) (1-X)^{n+1}.$$
We compute the four first derivatives of this polynomial at $x=1$ following Leibniz' rule to get the result, taking into account that $\chi (K) = 0$ when $n$ is odd.
\qed

\begin{rem}
By the simplest Dehn-Sommerville relation, every ridge of $K$ is contained in exactly two facets.  Since $\tau$ is a partition of $\vert K \vert$, the total number of ridges of all tiles of $\tau$ coincides with the total number of missing ridges of all tiles.  This implies the first relation given by Corollary \ref{corhvector}. This fact was already used in \cite{Wel1} to prove that the $h$-vector of a Morse tiling in dimension three is palindromic iff its $c$-vector is. 
\end{rem}

\begin{ex}
1) When $n=3$, Corollary \ref{corhvector} reads $h_3 (\tau) - h_1 (\tau) = 2(h_0 (\tau) - h_4 (\tau))$.

2) When $n=4$, the first two relations given by Corollary \ref{corhvector} read 
$$\big(h_3 (\tau) - h_2 (\tau) \big) + 3 \big( h_4 (\tau) - h_1 (\tau) \big) = 5 \big(h_0 (\tau) - h_5 (\tau) \big).$$
\end{ex}

\bibliographystyle{abbrv}

Univ Lyon, Universit\'e Claude Bernard Lyon 1, CNRS UMR 5208, Institut Camille Jordan, 43 blvd. du 11 novembre 1918, F-69622 Villeurbanne cedex, France

{\tt welschinger@math.univ-lyon1.fr.}
\end{document}